\documentclass[letterpaper, 10 pt, conference]{ieeeconf}  

\IEEEoverridecommandlockouts                              

\usepackage{cite}
\usepackage{amsmath,amssymb,amsfonts}
\usepackage{algorithmic}
\usepackage{graphicx}
\usepackage{textcomp}
\usepackage{xcolor}
\usepackage{thmtools}
\usepackage{mathtools}
\usepackage{mdframed}

\newtheorem{defn}{Definition}
\newtheorem{thm}{Theorem}
\newtheorem{lemma}{Lemma}
\newtheorem{cor}{Corollary}

\newcommand{\ev}{\mathcal{E}}
\newcommand{\Rd}{\mathbb{R}^d}
\newcommand{\knl}{\mathfrak{K}}
\newcommand{\spand}[1]{\text{span}\{\mathfrak{K}(\cdot,x):x\in #1\}}
\newcommand{\spanc}[1]{\overline{\text{span}\{\mathfrak{K}(\cdot,x):x\in #1\}}}
\newcommand{\Ho}{H_{\Omega}}
\newcommand{\Hn}{H_{\Omega_n}}
\newcommand{\Restr}[1]{\mathbf{R}_{#1}}
\newcommand{\Exten}[1]{\mathbf{E}_{#1}}
\newcommand{\wsub}[1]{_{W^{#1,2}(\Omega)}}

\title{\bf
Approximations of the Reproducing Kernel Hilbert Space (RKHS) Embedding Method over Manifolds
}

\author{Jia Guo$^{1}$, Sai Tej Paruchuri$^{1}$ and Andrew J. Kurdila$^{1}$
\thanks{$^{1}$Jia Guo, Sai Tej Paruchuri and Andrew J. Kurdila are with the Department of Mechanical Engineering, Virginia Tech, Blacksburg VA 24060, USA
        {\tt\small jguo18@vt.edu, saitejp@vt.edu, kurdila@vt.edu}}}

\begin{document}

\maketitle
\thispagestyle{empty}
\pagestyle{empty}

\begin{abstract}
The reproducing kernel Hilbert space (RKHS) embedding method is a recently introduced estimation approach that  seeks to identify  the unknown or uncertain function in the governing equations of a nonlinear set of ordinary differential equations (ODEs). While the original state estimate evolves in Euclidean space, the  function estimate is constructed in an infinite dimensional RKHS that must be approximated in practice. When  a  finite dimensional approximation is constructed using a basis defined in terms of  shifted kernel functions centered at the observations along a trajectory, the RKHS embedding method can be understood as a  data-driven approach. This paper derives sufficient conditions that ensure that  approximations of the unknown function converge in  a Sobolev norm over a submanifold that supports the dynamics. Moreover, the rate of convergence for the finite dimensional approximations is derived in terms of the fill distance of the samples in the embedded manifold.  A  numerical simulation of an example problem  is carried out to illustrate the qualitative nature of convergence results derived in the paper.

\end{abstract}

%
%
\section{Introduction}
\label{sec:intro} 
Data-driven modeling of uncertain or unknown nonlinear dynamic systems has been a topic of great interest over the past few years. This field synthesises  methods from dynamical system theory, estimation and control theory, approximation theory, and operator theory  and uses observations of states or observables to estimate quantities associated with an  unknown dynamic system \cite{klus2018,pan2018,liu2017}. The collection of algorithms that can, in some sense, be viewed as data-dependent methods is vast. Specific examples include the following: the collection of studies on the extended dynamic mode decomposition (EDMD) algorithm and its variants that are based on Koopman theory \cite{giannakis2019,drmac2018,macesic2018,bollt2018,li2017,korda2018,bruntonbook}; adaptive basis methods for online adaptive estimation \cite{pofar,zhao2007self,Chen2016Observer};  fuzzy control methods  based on neural networks\cite{san2019artificial,Chen2016Adaptive}; and strategies from distribution-free learning theory and nonlinear regression \cite{gyorfi}. 

Recently the authors have introduced a novel approach, the RKHS embedding method in \cite{Bo1,Bo2,Bo3,kl2013},  for the estimation of uncertain systems. This method likewise can be viewed as a type of  data-dependent algorithm when bases of approximation are selected along the trajectory of an unknown system. The RKHS embedding method generalizes  estimators used in conventional adaptive estimation over finite dimensional state spaces. The approach  essentially lifts the learning law of the estimation scheme  to an infinite dimensional RKHS $H$ of real-valued functions defined over the  state space. The unknown function $f(\cdot)$ that characterizes the uncertainty in the ordinary differential equations (ODEs) of dynamical system of is assumed to be an element of the RKHS $H$. The resulting overall estimator is thereby defined for both the states and the unknown function,  and it defines an  evolution in $\Rd\times H$. Since the evaluation functional $\ev_x$ is linear and bounded in the RKHS $H$, the unknown nonlinearity defined by $x\mapsto f(x)$ in the original ODE can be expressed as $\ev_x f$ in the RKHS embedding formulation, which is essentially a linear operator acting on the function $f\in H$. In this way, the nonlinearity in the original ODEs is avoided. The trade-off is that one has to conduct analysis in the infinite dimensional spaces, which is usually (much) more complicated.
   
Those familiar with the general philosophy  of Koopman theory will recognize a qualitative similarity between the advantages of the RKHS embedding method and approaches based on  Koopman theory. Both seek to replace a nonlinear system with one that is linear, and both must address the issues regarding convergence of approximations in coordinate realizations.
One significant difference might be that the RKHS method defines an estimator in \textit{continuous time} via an infinite dimensional distributed parameter system (DPS). As discussed in detail in \cite{kurdila2018koopman}, it is possible to further relate the two methods by showing that particular entries of the operator matrix that defines the DPS in RKHS embedding can be identified with some types of Koopman operators. 

In a way that is analogous to the conventional adaptive estimation in finite dimensional spaces, the convergence of the RKHS embedded estimator can be guaranteed with the satisfaction of a  condition of persistence of excitation (PE). The notion of the PE condition for the RKHS embedding method has been introduced and studied  in the authors' previous work \cite{gpk2019,kgp2019}. Given a subset $\Omega\subseteq \Rd$, a necessary condition for $\Omega$ to be PE by a positive orbit $\Gamma^+(x_0)$ starting at $x_0$  is that all the neighborhoods of  points in $\Omega$ are  ``visited infinitely often'' by the  trajectory. This means that $\Omega$ must be a subset of the $\omega$-limit set, in which every point is the limit of a subsequence of points extracted from the trajectory \cite{kgp2019}. A convergence result in \cite{kgp2019} states that over the indexing set $\Omega$ that is persistently excited, the estimate of the  unknown function converges to the actual function \cite{gpk2019}. However, the RKHS estimate generated by the DPS lies in the infinite dimensional RKHS. In order to obtain estimates that can be computed in practice, a finite dimensional approximation of the RKHS embedded equations has to be implemented. This paper studies the approximation of the adaptive RKHS embedded estimator. 

Most of time, a discussion about the convergence of approximations in an RKHS can be transformed into a discussion about the operator $I-\mathbf{P}_{\Omega_n}$, where $\mathbf{P}_{\Omega_n}:H_X\rightarrow \Hn$ denotes the orthogonal projection onto the finite dimensional approximant subspace $\Hn$. Additional insight, or sometimes a finer analysis, can be obtained by interpreting this projection error in other well-known spaces.
It is known that many commonly used RKHS are either embedded in or equivalent to some Sobolev space $W^{\tau,2}(\mathbb{R}^d)$. The fact that the family of Sobolev spaces provides a refined characterization of the smoothness of functions is useful for estimating the approximation error. In this paper, we first review carefully the relationship between some types of RKHS and Sobolev spaces. Then the error equations for some type of RKHS embedded estimator are recast in Sobolev spaces to facilitate the error analysis. Using the recently introduced results on the Sobolev error bounds for the interpolation operator \cite{Fu,HNW2010,HNW2011}, the rate of convergence for the finite dimensional approximation of the RKHS embedding equations is derived in terms of the fill distance of the samples $\Omega_n$ with respect to the subset (or manifold) $\Omega$. The most succinct form of this new error bound states that the error decays like $\mathcal{O}(h_{\Omega_n,\Omega}^{s-\mu})$ where $s$ and $\mu$ are the smoothness indices that depend on the choice of reproducing kernel and Sobolev spaces.

\subsection{Notations}
In this paper, the real-valued, symmetric, positive definite kernels over $X\times X$ are denoted by $\knl(\cdot,\cdot)$. The kernel basis function centered at $x\in X$ is denoted by $\knl_x(\cdot):=\knl(x,\cdot)$. The RKHS induced by $\knl$ is denoted by $H_X$. By $W^{s,2}(\Omega)$ we denote the Sobolev space defined on the domain $\Omega$. The restriction operator is denoted by $\Restr{\Omega}$, which restricts the domain of a  function to the subset $\Omega$. The space $\Restr{\Omega}H_X$ denotes the restriction to $\Omega$ of all of the functions $f\in H_X$.

%
%
\section{Problem Setup}
\label{sec:rkh_embedding}
%
%
\subsection{Reproducing Kernel Hilbert Spaces}
A real RKHS is a Hilbert space $H_X$ of real-valued functions defined over $X$ that admits a reproducing kernel $\knl:X\times X\rightarrow \mathbb{R}$. The kernel $\knl(\cdot,\cdot)$ has reproducing property provided for all $x\in X$ and $f\in H_X$, $f(x) = (f,\knl_x)_{H_X}$. The RKHS $H_X$ is the closure of all the linear spans of kernel basis functions centered at $x\in X$,
\begin{equation*}
    H_X = \spanc{X}.
\end{equation*}
Here the set of kernel centers $X$ is called the index set of $H_X$. If $\Omega\subseteq X$ is a subset, then $\Ho = \spanc{\Omega}$ is also an RKHS. Moreover, $\Ho$ is a subspace of $H_X$. The whole space $H_X$ can be expressed as the direct sum $H_X = \Ho \oplus V_\Omega$. Since any function $\phi\in V_\Omega$ is perpendicular to the space $\Ho$, the following condition is true
\begin{equation}
\label{eq:directSum}
    (\phi,\knl_x)_{H_X} = \phi(x) = 0 \quad \text{for all }x\in\Omega.
\end{equation}
For $\phi\in H_X$ to be in $V_\Omega$, it is necessary that $\phi(x)=0$ for all $x\in\Omega$. In fact, this condition is also sufficient \cite{Ber}. 

The condition in Eq. \ref{eq:directSum} is particularly useful for computing the projection $\mathbf{P}_\Omega f$ for $f\in H_X$. Since $(I-\mathbf{P}_\Omega)f\in V_\Omega$, it is sufficient and necessary to have for all $x\in\Omega$,
\begin{equation*}
    ((I-\mathbf{P}_\Omega)f,\knl_x)_{H_X} = f(x) - (\mathbf{P}_\Omega f)(x) = 0.
\end{equation*}
Thus for any \textit{finite discrete} set $\Omega_n$, the projection operation is in fact equivalent to the interpolation operation, $\mathbf{P}_{\Omega_n}\equiv I_{\Omega_n}:H_X\rightarrow \Hn$.

In this paper, we only consider the RKHS uniformly embedded in the space of continuous functions $C(X)$. As a result, if we let $\ev_x:f\mapsto f(x)$ denote the evaluation operator on $H_X$, then it can be deduced that $\ev_x$ is a bounded linear operator. One sufficient condition for $H_X$ to be continuously embedded in $C(X)$ is that there exists a constant $\bar{k}$ such that $\sqrt{\knl(x,x)}\leq \bar{k}<\infty$ for all $x\in X$. In fact, by the Cauchy-Schwartz inequality we have
{
\small
\begin{equation*}
    |\ev_x f| = |(f,\knl_x)_{H_X}| \leq \|f\|_{H_X}\|\knl_x\|_{H_X} = \|f\|_{H_X}\sqrt{\knl(x,x)}.
\end{equation*}
}
If the above constant $\bar{k}$ exists, then we have
{
\small
\begin{equation*}
    \|f\|_{C(X)} = \sup_{x\in X}|\ev_x f|\leq \|f\|_{H_X}\sup_{x\in X}\sqrt{\knl(x,x)}\leq \bar{k}\|f\|_{H_X},
\end{equation*}
}
which implies $H_X\hookrightarrow C(X)$. In all the following discussions, we assume such constant $\bar{k}$ always exists.

In some cases it is useful to consider the space of restrictions $\Restr{\Omega}H_X$. Clearly, the restriction operator $\Restr{\Omega}:H_X\rightarrow\Restr{\Omega}H_X$ is linear and onto. It follows from Eq. \ref{eq:directSum} that $\Restr{\Omega}V_\Omega = \{0\}$. In fact, one can show that $\Restr{\Omega}$ is a bijection over the subspace $\Ho$. To see why this is true, suppose $f_1,f_2\in\Ho$ are different functions. If $\Restr{\Omega} f_1 - \Restr{\Omega} f_2 = \Restr{\Omega}(f_1-f_2) = 0$, then $f_1-f_2\in V_\Omega$ because $(f_1-f_2,\knl_x)_{H_X} = 0$ for all $x\in\Omega$. But $f_1-f_2\in\Ho$ because $\Ho$ is a subspace, so $f_1-f_2=0$, which contradicts the assumption. With $\Restr{\Omega}|_{\Ho}$ being a bijection, the inverse operator $(\Restr{\Omega}|_{H_\Omega})^{-1}:\Restr{\Omega}H_X \rightarrow \Ho$ is well defined. We denote the inverse as the extension operator $\Exten{\Omega}:=(\Restr{\Omega}|_{\Ho})^{-1}$. It follows that
\begin{equation}
    \label{eq:P_ER}
    \mathbf{P}_{\Omega} = \Exten{\Omega}\Restr{\Omega}:H_X\rightarrow \Ho.
\end{equation}
By defining the inner product in $\Restr{\Omega}H_X$ as $(f,g)_{\Restr{\Omega}H_X}:=(\Exten{\Omega}f,\Exten{\Omega}g)_{H_X}$, we can show that $\Restr{\Omega}H_X$ is an RKHS. The associated reproducing kernel is shown to be the restriction of the reproducing kernel $\knl(\cdot,\cdot)$ over $\Omega\times\Omega$, that is, $\Restr{\Omega}\knl:=\knl|_{\Omega\times\Omega}$ \cite{Fu}.

%
%
\subsection{RKHS embedded estimator}

We assume the governing equation of a partially unknown dynamic system can be written as follows
\begin{equation}
    \dot{x}(t) = Ax(t) + Bf(x(t)), \label{eq:orig_dyn}
\end{equation}
where $A\in\mathbb{R}^{d\times d}$ is a Hurwitz matrix, $B\in\mathbb{R}^{d\times 1}$, and $f:\Rd\rightarrow\mathbb{R}$ is the unknown nonlinear function. As discussed more fully in \cite{gpk2019}, this equation suffices for the proofs that follow, and more general governing equations can be treated as discussed in this reference. The governing equations of the RKHS embedded estimator contain one equation for the state estimates and another for the function estimates,
\begin{equation}
\label{eq:dyn_inf}
\begin{aligned}
    \dot{\hat{x}}(t) &= A\hat{x}(t) + B\ev_{x(t)}\hat{f}(t), \\
    \dot{\hat{f}}(t) &= \gamma (B\ev_{x(t)})^*P(x(t) - \hat{x}(t)),
\end{aligned}
\end{equation}
where $\hat{f}(t)\in H_X$ is the time-varying function estimate in the RKHS, $\ev_{x(t)}$ denotes the evaluation operator, and $P$ is the solution to the Lyapunov equation associated to the Hurwitz matrix $A$. That is, $P$ satisfies $A^TP+PA=-Q$ for some positive definite $Q\in\mathbb{R}^{d\times d}$. Assuming $f\in H_X$, we can replace the nonlinear term $f(x(t))$ with $\ev_{x(t)}f$ in Eq. \ref{eq:orig_dyn}. Denote the estimation errors by $\tilde{x}(t) = x(t) - \hat{x}(t)$ and $\tilde{f}(t) = f - \hat{f}(t)$. The estimation error equations can be written as follows,
\begin{equation}
\label{eq:dyn_inferr} 
\begin{aligned}
    \dot{\tilde{x}}(t) &= A\tilde{x}(t) + B\ev_{x(t)}\tilde{f}(t), \\
    \dot{\tilde{f}}(t) &= -\gamma (B\ev_{x(t)})^*P\tilde{x}(t),
\end{aligned}
\end{equation}
which resembles its counterpart in conventional adaptive estimation, but evolves in infinite dimensional space $\Rd\times H_X$. It has been proven that the equilibrium the system expressed in Eq. \ref{eq:dyn_inferr} at the origin is uniformly asymptotically stable (UAS) if the PE condition is satisfied. See \cite{gpk2019,kgp2019} for discussion about the PE condition and the stability of the error dynamics. 

Practical implementations of the RKHS embedded estimator require the construction of approximations in finite dimensional subspaces $\Rd\times H_{\Omega_n}$. We denote the states of the finite-dimensional estimator with $(\hat{x}_n(t),\hat{f}_n(t))$. The approximation $\hat{f}_n(t)$ lies in the approximant space $H_{\Omega_n}$ defined as 
\begin{equation*}
    H_{\Omega_n} = \spand{\Omega_n} \subseteq H_X
\end{equation*}
for some finite collections of distinct centers $\Omega_n = \{\xi_i\}_{i=1}^n$. In this paper we explore the case when the centers are taken from the positive orbit $\Gamma^+(x_0) = \bigcup_{t\geq 0}x(t)$ of the uncertain system. This choice makes the approach a data-driven method. The evolution equations of the finite dimensional estimator are written as 
\begin{equation}
\label{eq:dyn_approx}
    \begin{aligned}
    \dot{\hat{x}}_n(t) &= A\hat{x}_n(t) + B\ev_{x(t)}\mathbf{P}_{\Omega_n}^*\hat{f}_n(t), \\
    \dot{\hat{f}}_n(t) &= \gamma \mathbf{P}_{\Omega_n} (B\ev_{x(t)})^*P(x(t) - \hat{x}_n(t)).
    \end{aligned}
\end{equation}
We call the difference $\tilde{f}_n(t)=f - \hat{f}_n(t)$ the total error of the function estimate. It is the summation of two parts, the ``infinite dimensional'' error $\tilde{f}(t) = f - \hat{f}(t)$ and the finite dimensional approximation error $\bar{f}(t) = \hat{f}(t) - \hat{f}_n(t)$. The corresponding notations $\tilde{x}(t)$ and $\bar{x}(t)$ are defined accordingly for the state estimate. From Eq. \ref{eq:orig_dyn}, \ref{eq:dyn_inf} and \ref{eq:dyn_approx}, we can derive the dynamical equations as follows that characterize the evolution of the approximation error,
\begin{align}
    \dot{\bar{x}}_n(t) &= A\bar{x}_n(t) + B\ev_{x(t)}\bar{f}_n(t), \label{eq:err_xbar}\\
    \dot{\bar{f}}_n(t) &= -\gamma\mathbf{P}_{\Omega_n}(B\ev_{x(t)})^*P\bar{x}_n(t) \nonumber \\
        &\hspace{4em} + \gamma(I-\mathbf{P}_{\Omega_n})(B\ev_{x(t)})^*P\tilde{x}(t). \label{eq:err_fbar}
\end{align}

%
%
\subsection{Sobolev Spaces}
For a subset $\Omega\subseteq X = \Rd$ and a positive integer $s$, the Sobolev space $W^{s,2}(\Omega)$ is the collection of functions in $L^2(\Omega)$ that have weak derivatives of all orders less than or equal to $s$ that are also in $L^2(\Omega)$. The space $W^{s,2}(\Omega)$ is equipped with the norm
\begin{equation}
    \|f\|\wsub{s}^2 = \sum_{0\leq |\alpha| \leq s} \int_\Omega \left\vert \frac{\partial^\alpha f}{\partial x^\alpha} \right\vert^2 d\mu,
\end{equation}
where $\alpha = (\alpha_1,...,\alpha_d)$ denotes the multi-indices of derivative with $\sum_{i=1}^d \alpha_i = |\alpha|$. The integral is taken with respect to the Lebesgue measure. The Sobolev spaces for non-integer $t>0$ are defined in terms of interpolation space theory. Similar definitions hold for Sobolev spaces $W^{s,2}(\Omega)$ over a manifold $\Omega$, with the derivatives replaced by covariant intrinsic derivatives on $\Omega$ and $\mu$ the volume measure of the manifold. See \cite{Adams} for details.

As discussed above, we must consider the restrictions of functions to subsets. One particularly important case needed in this paper is when $\Omega\subseteq X$ is a $k$-dimensional smooth compact embedded manifold. By trace theorem (Proposition 2 in \cite{Fu}), if $\tau>(d-k)/2$, then we have 
\begin{equation}
\label{eq:RWX_WO}
    \Restr{\Omega}W^{\tau,2}(X) \approx W^{\tau-(d-k)/2,2}(\Omega)
\end{equation}
with equivalent norms. The important implication in the above equivalence is that an amount of ``smoothness'' is lost due to the restriction operation onto the low-dimensional smooth submanifold, which will affect the convergence rates in this paper.

Many Sobolev spaces are also reproducing kernel Hilbert spaces. Let $\knl:X\times X\rightarrow\mathbb{R}$ be a reproducing kernel and $\hat{\knl}:=\mathcal{F}(\knl)$ be the Fourier transform of $\knl$. Suppose $\hat{\knl}$ has algebraic decay as follows
\begin{equation*}
    \hat{\knl}(\xi) \sim (1+\|\xi\|_2)^{-\tau}, \quad \tau>d/2.
\end{equation*}
Then it has been shown in \cite{Fu,Ber} that the RKHS $H_X$ induced by $\knl$ itself is a Sobolev space with smoothness index $\tau$. Specifically, we have
\begin{equation}
\label{eq:HX_WX}
    H_X\approx W^{\tau,2}(X)
\end{equation}
with equivalent norms. Following the arguments of Theorem 5 in \cite{Fu} the equivalence displayed in Eq. \ref{eq:HX_WX} can be extended to the space of restrictions. When $\Omega\subseteq X$ is the submanifold defined above, then we have
\begin{equation}
    \Restr{\Omega}H_X \approx \Restr{\Omega}W^{\tau,2}(X) \approx W^{\tau-(d-k)/2,2}(\Omega).
\end{equation}

%
%
\section{Main Results}
\label{sec:main_results}
\subsection{Restriction of Evolution on Manifolds}
Ultimately we are interested in restricting the evolution equations governed by the DPS over $X$ to a subset $\Omega\subseteq X$. This section shows that 
$$
\frac{d}{dt}\left (\Restr{\Omega} h(t) \right ) = \Restr{\Omega}\left (\frac{d}{dt} h(t)  \right )
=\Restr{\Omega} \dot{h}(t),
$$
where the derivative is understood in the strong sense. 
We begin with a lemma that shows that $\Restr\Omega\dot{h}(t)=\frac{d}{dt}(\Restr{\Omega} h(t))$ where $\frac{d}{dt}$ in the proof is temporarily understood as the weak derivative with respect to time.
\begin{lemma}
\label{lm:commute}
Let $h\in C^{1}([0,T],H_\Omega)$. Then,
$$
\Restr{\Omega} \dot{h}(t) = \frac{d}{dt}(\Restr{\Omega} h(t))
$$
with $\frac{d}{dt}$ the weak derivative in time. 
\end{lemma}
\begin{proof}
A mapping $h:[0,T]\rightarrow H_\Omega$ is weakly differentiable if there is a distribution $v=\dot{h}:[0,T]\rightarrow H_\Omega$ such that 
\begin{align}
\int_0^T (h(\zeta),\dot{\phi}(\zeta)w)_{H_\Omega} d\zeta=-\int_0^T(v,\phi(\zeta)w)_{H_\Omega}d\zeta
\label{eq:weakDef}
\end{align}
for all $w\in H_\Omega$ and $\phi\in C_0^\infty[0,T]$ \cite{Zei}. From Section \ref{sec:rkh_embedding} we know that the projection operator can be expressed as $\mathbf{P}_\Omega  =\Exten{\Omega}\Restr{\Omega}$. Note that for $f\in\Ho$ the projection operator acts as the identity, so the left hand side of Eq. \ref{eq:weak1} can be written as 
\begin{align}
&\int_0^T(\mathbf{P}_\Omega h(\zeta), \mathbf{P}_\Omega \dot{\phi}(\zeta)w)_{H_X} d\zeta \nonumber \\
&\hspace{4em} = \int_0^T(\Exten{\Omega}\Restr{\Omega} h(\zeta), \Exten{\Omega}\Restr{\Omega} \dot{\phi}(\zeta)w)_{H_X} d\zeta, \nonumber \\
&\hspace{4em} = \int_0^T(\Restr{\Omega} h(\zeta),  \Restr{\Omega}\dot{\phi}(\zeta)w)_{\Restr{\Omega}H_X} d\zeta. \label{eq:weak1}
\end{align}
The right hand side of Equation \ref{eq:weak1} can similarly be expressed as 
\begin{align}
&\int_0^T(v(\zeta),  \phi(\zeta)w)_{H_\Omega} d\zeta, \nonumber \\
&\hspace{4em} = \int_0^T(\Exten{\Omega}\Restr{\Omega}v(\zeta), \phi(\zeta)\Exten{\Omega}\Restr{\Omega}w(\zeta))_{H_X} d\zeta, \nonumber \\
&\hspace{4em} =\int_0^T(\Restr{\Omega} v(\zeta),  \phi(\zeta)\Restr{\Omega}w)_{\Restr{\Omega}H_X} d\zeta. \label{eq:weak2}
\end{align}
Combining Equations \ref{eq:weak1} and \ref{eq:weak2}, we have
\begin{align}
&\int_0^T( \Restr{\Omega} h(\zeta), \dot{\phi}(\zeta) \Restr{\Omega} w)_{\Restr{\Omega}(H_X)} d\zeta \nonumber \\
&\hspace{4em} = -\int_0^T( \Restr{\Omega} v(\zeta), \phi(\zeta) \Restr{\Omega} w)_{\Restr{\Omega}(H_X)} d\zeta
\end{align}
for all $\phi\in C_0^\infty[0,T]$ and $w\in H_\Omega$. Since the operator $\Restr{\Omega}$ is onto $\Restr{\Omega}H_X$, the conclusion of the lemma follows.
\end{proof}

Since $\Restr{\Omega}$ is a bounded linear operator, we also know that 
\begin{align*}
    &\frac{\| (\Restr{\Omega} f)(t+\Delta) - (\Restr{\Omega} f)(t) \|_{\Restr{\Omega}(H_X)}}{\Delta} \\
    &\hspace{2em} \leq \|\Restr{\Omega}\| \frac{\|f(t+\Delta)-f(t)\|_{H_X}}{\Delta}\rightarrow \|\Restr{\Omega}\| \|\dot{f}(t)\|_{H_X}, 
\end{align*}
which implies that the map $t\mapsto (\Restr{\Omega} f)(t)$ is strongly differentiable.  This means that $\Restr{\Omega} \dot{f}=\frac{d}{dt}(\Restr{\Omega} f)$ where now the both derivatives are interpreted in the strong sense. 

With the conclusion from Lemma \ref{lm:commute}, we are able to derive the evolution of the finite dimensional approximation restricted to the manifold $\Omega$. Apply the restriction operator $\Restr{\Omega}$ to both sides of Eq. \ref{eq:err_fbar}, and the error equation with respect to the restriction $\Restr{\Omega}\bar{f}_n(t)$ is 
\begin{align}
\label{eq:err_rfbar}
&\frac{d}{dt}\left(\Restr{\Omega}\bar{f}_n(t)\right) = \Restr{\Omega}\dot{\bar{f}}_n(t) \nonumber \\
 &\hspace{3em}=-\gamma\Restr{\Omega}\mathbf{P}_{\Omega_n}(B\ev_{x(t)})^*P\bar{x}_n(t) \nonumber \\
 &\hspace{5em}+ \gamma\Restr{\Omega}(I-\mathbf{P}_{\Omega_n})(B\ev_{x(t)})^*P\tilde{x}(t).
\end{align}


\subsection{Error Estimates of Sobolev Type}
In this section we derive the primary theoretical result of this paper that is summarized in Theorem \ref{th:sobolev} and Corollary \ref{cor:sobolevrate}. First we review some results of an analysis in \cite{Nar2005,Ar2007} about Sobolev error bounds for scattered data interpolation. Let $t$ and $\mu$ be the orders of two Sobolev spaces of functions defined over a smooth manifold $\Omega$. Clearly, given $t>\mu$, it follows that $W^{t,2}(\Omega)\subseteq W^{\mu,2}(\Omega)$. Suppose the function $u$ has a set of zero points $\Omega_n\subset\Omega$ (i.e. $u|_{\Omega_n}=0$)  which are distributed densely enough in $\Omega$. This theorem characterizes the relationship between $\|u\|\wsub{t}$ and $\|u\|\wsub{\mu}$ in the term of the fill distance of zeros $h_{\Omega_n,\Omega}$, which is defined below.

\begin{defn}[Fill Distance \cite{Fu}]
For a finite set of discrete points $\Omega_n=\{\xi_i\}_{i=1}^n$ in a metric space $\Omega$, the fill distance $h_{\Omega_n,\Omega}$ of $\Omega_n$ with respect to $\Omega$ is defined as
\begin{equation*}
    h_{\Omega_n,\Omega}:=\sup_{x\in\Omega}\min_{\xi_i\in\Omega_n} d(x,\xi_i),
\end{equation*}
where $d(\cdot,\cdot)$ is the intrinsic metric in $\Omega$.
\end{defn}
In the case which is of the most interest to this paper, the set $\Omega$ is a compact smooth Riemannian submanifold in $\Rd$, and the discrete set $\Omega_n$ is the set of interpolation points in the manifold. With this definition in mind, the following theorem states the relationship between $\|u\|\wsub{t}$ and $\|u\|\wsub{\mu}$.
\begin{thm}
\label{th:many0}
Let $\Omega\subseteq\Rd$ be a smooth $k$-dimensional manifold, $t\in \mathbb{R}$ with $t>k/2$, $\mu\in\mathbb{N}_0$ with  $0\leq \mu \leq \lceil t\rceil -1$.Then there is a constant $h_\Omega$ such that if the fill distance $h_{\Omega_n,\Omega}\leq h_\Omega$ and $u\in W^{t,2}(\Omega)$ satisfies $u|_{{\Omega_n}}=0$, then
\begin{equation*}
\|u\|_{W^{\mu,2}(\Omega)}\lesssim h_{\Omega_n,\Omega}^{t-\mu}\|u\|_{W^{t,2}(\Omega)}.    
\end{equation*}
\end{thm}

We denote the interpolation operator over $\Omega_n$ by $I_{\Omega_n}$. For a function $f\in W^{t,2}(\Omega)$, the interpolation error $(I-I_{\Omega_n})f$ by definition has zeros over $\Omega_n$. A corollary is introduced in \cite{Fu} to characterize the decaying rate of interpolation error in the native space (i.e. RKHS). 
\begin{cor}
\label{cor:rkhs_bound}
Let $\Omega \subseteq X:=\Rd$ be a $k$-dimensional smooth manifold, and let the native space $H_X$ be continuously embedded in a Sobolev space $W^{\tau,2}(X)$
with $\tau>d/2$, so that $\|f\|_{W^{\tau,2}(\mathbb{R}^d)} \lesssim\|f\|_{H_X}. $ Define $s=\tau-(d-k)/2$ and let $0\leq \mu \leq \lceil s\rceil -1$. Then there is a constant $h_\Omega$ such that if $h_{\Omega_n,\Omega}\leq h_\Omega$, then for all $f\in \Restr{\Omega}(H_X)$ we have 
\begin{equation*}
  \|(I-I_{\Omega_n})f\|\wsub{\mu} \lesssim h_{\Omega_n,\Omega}^{s-\mu}\|f\|_{\Restr{\Omega}(H_X)}.  
\end{equation*}
\end{cor}

With this error bound in mind for interpolation and projection, we now turn to the main result of this paper.
\begin{thm}
\label{th:sobolev}
Suppose that $\Omega$ is a compact, connected, $k$-dimensional, regularly embedded Riemannian submanifold of $X:=\mathbb{R}^d$, $\Restr{\Omega}(H_X) \hookrightarrow  W^{s,2}(\Omega)$ for some $s>k/2$, and the orbit of the unknown system $\Gamma^+(x_0)\subseteq \Omega$. Then there exist two constants $a,b>0$ such that for all $t\in[0,T]$, 
\begin{align*}
    &\|\bar{x}_n(t)\|^2 + \|\Restr{\Omega}\bar{f}_n(t)\|\wsub{s}^2 \\ 
    &\hspace{2em}\leq e^{bt}\bigg(  \|(I-\Pi_{\Omega_n})\Restr{\Omega}\hat{f}_0\|^2\wsub{s} \\
    &\hspace{4em}+ a\int_0^t \|(I-\Pi_{\Omega_n})\Restr{\Omega}\knl_{x(\zeta)}\|\wsub{s}^2 d\zeta \bigg).
\end{align*}
Here $\Pi_{\Omega_n}:\Restr{\Omega}H_X\rightarrow\Restr{\Omega}\Hn$ denotes the projection operator defined over the space of restriction $\Restr{\Omega}H_X$ onto the space $\Restr{\Omega}\Hn$. 
\end{thm}
\begin{proof}
The proof of the inequality below uses the fact that $\Restr{\Omega} \left (\frac{d}{dt}h(t) \right )=\frac{d}{dt}(\Restr{\Omega} h(t))$ where the time derivative is understood in the strong sense, i.e. $\Restr{\Omega}\dot{h}(t)=\frac{d}{dt}(\Restr{\Omega} h(t))$. By Eq. \ref{eq:err_xbar} and \ref{eq:err_rfbar} we have
{
\small
\begin{align*}
    &\frac{d}{dt} \left (\|\bar{x}_n(t)\|^2 + \|\Restr{\Omega}\bar{f}_n(t)\|^2\wsub{s}{s} \right) = (\dot{\bar{x}}_n(t),\bar{x}_n(t)) \\
      &\hspace{2em}+ (\bar{x}_n(t),\dot{\bar{x}}_n(t)) + 2\left(\Restr{\Omega}\bar{f}_n(t), \frac{d}{dt}(\Restr{\Omega}\bar{f}_n(t))\right)\wsub{s} \\
     &\hspace{1em} = ((A+A^T)\bar{x}_n(t),\bar{x}_n(t)) + \underbrace{2(B\ev_{x(t)}\bar{f}_n(t),\bar{x}_n(t))}_{\text{term 1}} \\
      &\hspace{2em} + \underbrace{2\gamma\left (\Restr{\Omega} (I-\mathbf{P}_{\Omega_n})\ev_{x(t)}^*B^TP\tilde{x}(t),\Restr{\Omega} \bar{f}_n(t) \right)\wsub{s}}_{\text{term 2}} \\
      &\hspace*{2em} + \underbrace{2\gamma\left(- \Restr{\Omega} \mathbf{P}_{\Omega_n}\ev_{x(t)}^*B^TP\bar{x}_n(t),\Restr{\Omega} \bar{f}_n(t) \right )\wsub{s}}_{\text{term 3}}.
\end{align*}
}
\noindent 
In several of the steps that follow, we use conclusions that result from the assumptions that certain spaces are continuously embedded in others. By choosing certain types of reproducing kernels, it can be guaranteed that the injection $j:\Restr{\Omega}(H_X)\rightarrow W^{s,2}(\Omega)$ is such that we have the continuous embedding
$$
\Restr{\Omega}(H_X) \overset{j}{\hookrightarrow} W^{s,2}(\Omega)
$$
for some $s>k/2$. By the Sobolev embedding theorem the injection $i:W^{s,2}(\Omega)\rightarrow C(\Omega)$ is also continuous, that is 
$$
 W^{s,2}(\Omega)\overset{i}{\hookrightarrow} C(\Omega).
$$
This implies that 
$$
|\ev_{s,x}(f)|:=|f(x)|
\leq \|if\|_{C(\Omega)}\leq \|i\|\|f\|_{W^{s,2}(\Omega)},
$$
and it follows that each evaluation functional $\ev_{s,x}:W^{s,2}(\Omega)\rightarrow \mathbb{R}$ is uniformly bounded by $\|i\|$. By assumption we have that the forward orbit $\Gamma^+(x_0)\subseteq \Omega$,  we conclude that term 1 can be bounded by the expression 
\begin{align}
    &|(B\ev_{x(t)}\Restr{\Omega} \bar{f}_n(t),\bar{x}_n(t))| \notag \\
    &\hspace{2em} \leq \|B\|\|i\|\|\bar{x}_n(t)\| \|\Restr{\Omega} \bar{f}_n(t)\|_{W^{s,2}(\Omega)}. 
    \label{eq:wsterm1}
\end{align}
We bound term 2 by 
\begin{align}
    & \left|\left (\Restr{\Omega} (I-\mathbf{P}_{\Omega_n})\ev_{x(t)}^*B^TP\tilde{x}(t),\Restr{\Omega} \bar{f}_n(t) \right)\wsub{s}\right| \nonumber \\
    &\hspace{2em} \leq \|\Restr{\Omega} (I-\mathbf{P}_{\Omega_n})\ev_{x(t)}^*B^TP\tilde{x}(t)\|\wsub{s}\nonumber \\
    &\hspace{6em}\|\Restr{\Omega} \bar{f}_n(t)\|_{W^{s,2}(\Omega)}, \nonumber \\
    &\hspace{2em} \leq \|B\| \|P\|\|\tilde{x}(t)\|
    \|(I-\Pi_{\Omega_n})\Restr{\Omega} \mathfrak{K}_{x(t)} \|_{W^{s,2}(\Omega)} \nonumber \\
    &\hspace{6em}\|\Restr{\Omega}\bar{f}_n(t)\|_{W^{s,2}(\Omega)}.
\end{align}
We next consider term 3, which satisfies the inequality
\begin{align}
   &\left|\left (\Restr{\Omega} 
    \mathbf{P}_{\Omega_n}\ev_{x(t)}^*B^TP\bar{x}_n(t)  ,\Restr{\Omega} \bar{f}_n(t) \right )_{W^{s,2}(\Omega)}
    \right|  \nonumber \\ 
   &\hspace{2em} \leq \left \|j\Restr{\Omega} 
    \mathbf{P}_{\Omega_n}\ev_{x(t)}^*B^TP\bar{x}_n(t)  \right \|_{W^{s,2}(\Omega)} \nonumber \\
    &\hspace{6em} \times \left \| \Restr{\Omega} \bar{f}_n(t) \right \|_{W^{s,2}(\Omega)} ,
    \nonumber \\
   &\hspace{2em}\leq \bar{k}\|j\|\|\Restr{\Omega}\| \|B\|\|P\|\|\bar{x}_n(t)\| \nonumber \\ &\hspace{6em} \times \|\Restr{\Omega}\bar{f}_n(t)\|_{W^{s,2}(\Omega)}. 
\end{align}
Combining all the terms above, we obtain 
\begin{align*}
    &\frac{d}{dt}\left(\|\bar{x}_n(t)\|^2 + \|\Restr{\Omega} \bar{f}_n(t)\|\wsub{s}^2\right) \\
    &\hspace{1em}\leq \gamma \|B\|^2 \|P\|^2 \|\tilde{x}(t)\|^2 \|(I-\Pi_{\Omega_n})\Restr{\Omega}\knl_{x(t)}\|\wsub{s}^2 \\
    &\hspace{2em} +\big(2\|A\| + \|i\|\|B\|\big)\|\bar{x}_n(t)\|^2\\
    &\hspace{4em} + \big(\gamma\bar{k}^2\|j\|^2\|\Restr{\Omega}\|^2 \|B\|^2\|P\|^2\big) \|\bar{x}_n(t)\|^2 \\
    &\hspace{2em} +\big(2\gamma + \|i\|\|B\|\big)\|\Restr{\Omega}\bar{f}_n(t)\|\wsub{s}^2.
\end{align*}
Let the constants $a,b$ be defined as follows.
\begin{align*}
    a&:= \gamma\|B\|^2\|P\|^2\sup_{\zeta\in[0,T]}\|\tilde{x}(\zeta)\|^2, \\
    b&:= \max \{2\gamma + \|i\|\|B\|, \\
    &\hspace{2em} 2\|A\|+\|i\|\|B\| +\gamma \bar{k}^2 \|j\|^2 \|\Restr{\Omega}\|^2 \|B\|^2 \|P\|^2 \}.
\end{align*}
When we integrate  the inequality above,  it follows that
\begin{align*}
    &\|\bar{x}_n\|^2 + \|\Restr{\Omega}\bar{f}_n(t)\|\wsub{s}^2 \\
    &\hspace{1em}\leq \|\bar{x}_n(0)\|^2 + \|\Restr{\Omega}\bar{f}_n(0)\|\wsub{s}^2 \\
    &\hspace{2em} + \int_0^t a\|(I-\Pi_{\Omega_n})\Restr{\Omega}\knl_{x(\zeta)}\|\wsub{s} d\zeta \\
    &\hspace{2em} +\int_0^t\big(\|\bar{x}(t)\|^2 + \|\Restr{\Omega} \bar{f}_n(\zeta)\|\wsub{s}^2 \big) d\zeta.
\end{align*}
But since $\hat{f}_0\in H_\Omega$, we know that $\mathbf{P}_\Omega \hat{f}_0= \Exten{\Omega} \Restr{\Omega} \hat{f}_0=\hat{f}_0$, and $\Restr{\Omega} \Pi_{\Omega_n}\hat{f}_0=\Pi_{\Omega_n}\Restr{\Omega} \hat{f}_0$. Thus
\begin{align*}
    \|\Restr{\Omega} \bar{f}_n(0)\|_{W^{s,2}(\Omega)} = \|(I-\Pi_{\Omega_n})\Restr{\Omega}\hat{f}_0\|_{W^{s,2}(\Omega)}.
\end{align*}
Combining the above inequalities yields
\begin{align*}
&\|\bar{x}_n(t)\|^2 + \|\Restr{\Omega} \bar{f}_n(t)\|_{W^{s,2}(\Omega)}^2  \\
&\hspace{2em} \leq e^{bt}\|(I-\Pi_{\Omega_n})\Restr{\Omega} \hat{f}_0\|_{W^{s,2}(\Omega)}^2 \\
&\hspace{4em} + ae^{bt}\int_0^t \|(I-\Pi_{\Omega_n})\Restr{\Omega}\knl_{x(\zeta)}\|^2_{W^{s,2}(\Omega)}d\zeta.
\end{align*}
\end{proof}
The next corollary combines the results of Theorem \ref{th:sobolev} and Corollary \ref{cor:rkhs_bound} to obtain the error rates in terms of the fill distance of samples in the manifold. 
\begin{cor}
\label{cor:sobolevrate}
Suppose that $\Omega$ is a compact, connected, regularly embedded $k$-dimensional submanifold of $X:=\mathbb{R}^d$, the kernel $\mathfrak{K}$ is selected so that $H_X\hookrightarrow W^{\tau,2}(X)$ for $\tau>d/2$,  define  $s:=\tau -(d-k)/2$, and let $\mu\in [k/2,\lceil s  \rceil -1]$. Then we have
\begin{align*}
&\|\bar{x}_n(t)\|^2 + \|\Restr{\Omega} \bar{f}_n(t)\|_{W^{\mu,2}(\Omega)}^2  \\
&\hspace{4em} \leq 
\left (\|\Restr{\Omega} \hat{f}_0\|_{\Restr{\Omega}(H_X)}^2 + a\bar{k}^2 t  \right)e^{\tilde{b}t}
h^{2(s-\mu)}_{\Omega,\Omega_n}  .
\end{align*}
\end{cor}
\begin{proof}
We first observe that under the stated hypotheses the native space $\Restr{\Omega}H_X \hookrightarrow W^{\mu,2}(\Omega)$.  For any two positive $r_1\geq r_2>0$ the associated  Sobolev spaces are a continuous scale of spaces with 
$ W^{r_1,2}(\Omega)\hookrightarrow W^{r_2,2}(\Omega)$. This implies that $W^{s,2}(\Omega)\hookrightarrow W^{\mu,2}(\Omega)$ since $s\geq \mu$. Also, the trace theorem yields 
\begin{align*} 
\|f\|\wsub{\mu} \lesssim &\|f\|\wsub{s} = \|\Exten{\Omega}\Restr{\Omega}f\|\wsub{s} \\
    &\hspace{2em}\lesssim \|\Exten{\Omega} f\|_{W^{\tau,2}(X)} \lesssim \|f\|_{\Restr{\Omega}H_X},
\end{align*}
where the constants in the above string of inequalities depend on $\|\Exten{\Omega}\|$, $\|\Restr{\Omega}\|,$ and the norm of the embedding of $H_X$ into $W^{\tau,2}(X)$. 
Combining these results yields $\Restr{\Omega}H_X\hookrightarrow W^{\mu,2}(\Omega)$ with $\mu\geq k/2$. We can now apply the results of Theorem \ref{th:sobolev} for the choice $s=\mu$ and write
\begin{align*}
    &\|\bar{x}_n(t)\|^2 + \|\Restr{\Omega}\bar{f}_n(t)\|_{W^{\mu,2}(\Omega)}^2\\ 
   &\hspace{2em}\leq e^{bt}\|(I-\Pi_{\Omega_n})\Restr{\Omega}\hat{f}_0\|^2\wsub{\mu} \\
    &\hspace{4em} + ae^{bt}\|(I-\Pi_{\Omega_n})\Restr{\Omega}\knl_{x(t)}\|^2\wsub{\mu}, \\
   &\hspace{2em}\leq e^{bt}h^{2(s-\mu)}_{\Omega,\Omega_n}\|\Restr{\Omega}\hat{f}_0\|_{\Restr{\Omega}H_X}^2 \\
    &\hspace{4em} + ate^{bt} h^{2(s-\mu)}_{\Omega,\Omega_n}\left (\sup_{t\in [0,T]}\|\Restr{\Omega}\knl_{x(t)}\|_{\Restr{\Omega}H_X}\right)^2
\end{align*}
for each $t\in [0,T]$ by Theorem 11 of \cite{Fu}. The bound now follows since $$\sup_{t \in [0,T]}\|\Restr{\Omega}\knl_{x(t)}\|_{\Restr{\Omega}H_X}=\sup_{t\in [0,T]}\|\knl_{x(t)}\|_{H_X}\leq {\bar{k}}.$$ 
\end{proof}

%
%
\section{Numerical Simulation}
\label{sec:simulation}
Corollary \ref{cor:sobolevrate} gives a rate of convergence for finite dimensional approximations of the RKHS embedded estimator. The rate of convergence depends on the density of the sample set $\Omega_n$ in the manifold $\Omega$, which is characterized by the fill distance $h_{\Omega_n,\Omega}$. In this section, this rate of convergence is illustrated qualitatively using numerical simulations. Following the formulation of Eq. \ref{eq:orig_dyn}, the governing equations of the unknown system are selected to be
\begin{equation}
\label{eq:sim_dyn}
\begin{bmatrix}
\dot{x}_1(t) \\ \dot{x}_2(t)
\end{bmatrix} = \begin{bmatrix}
0 & 1 \\ -1 & 0
\end{bmatrix}\begin{bmatrix}
x_1(t) \\ x_2(t)
\end{bmatrix} + \begin{bmatrix}
x_1^2(t) \\ 0
\end{bmatrix},
\end{equation}
where $B=[1,0]^T$ and $f(x_1,x_2) = x_1^2$. Here we assume the linear coefficient matrix $A_0$ is known. By adding and subtracting a selected Hurwitz term $Ax(t)$, the governing equations of the original system can be written as 
\begin{equation*}
    \dot{x}(t) = Ax(t) + (A_0-A)x(t) + B\ev_{x(t)}f.
\end{equation*}
Since $A_0$ and $A$ are known, the term $(A_0-A)x(t)$ can be canceled in the error equation. The governing equations of the finite dimensional RKHS embedded estimator are chosen as
\begin{align*}
    \dot{\hat{x}}_n(t) &= A\hat{x}_n(t) + (A_0-A)x(t) + B\ev_{x(t)}\mathbf{P}^*_{\Omega}\hat{f}_n(t), \\
    \dot{\hat{f}}_n(t) &= \gamma \mathbf{P}_{\Omega} (B\ev_{x(t)})^*P(x(t)-\hat{x}_n(t)).
\end{align*}
This choice yields the error equations that have the form studied in this paper.
\begin{figure}[!h]
    \centering
    \includegraphics[width=\linewidth]{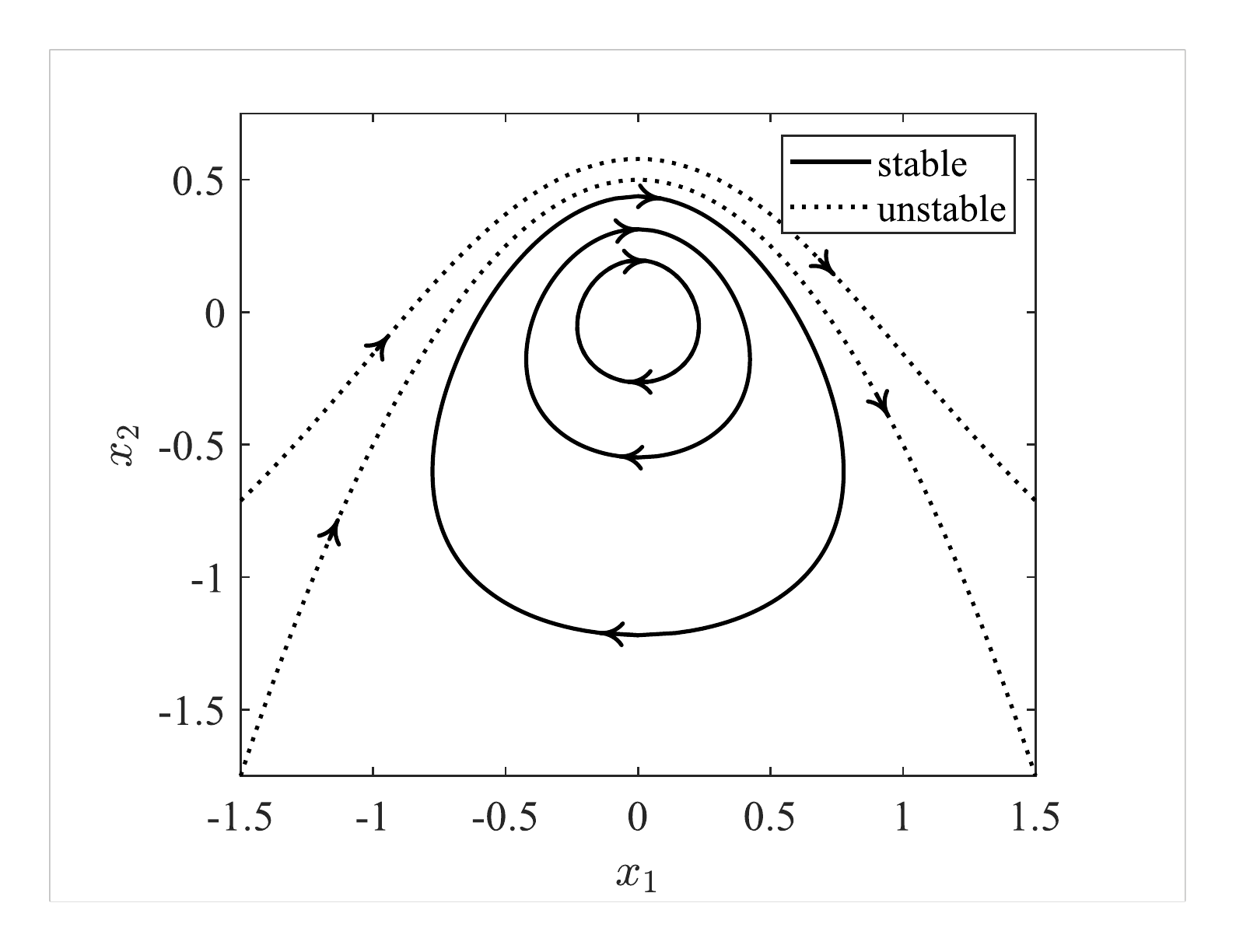}
    \caption{Phase trajectories of the actual system.}
    \label{fig:phase_plot}
\end{figure}
\begin{figure}[!h]
    \centering
    \includegraphics[width=\linewidth]{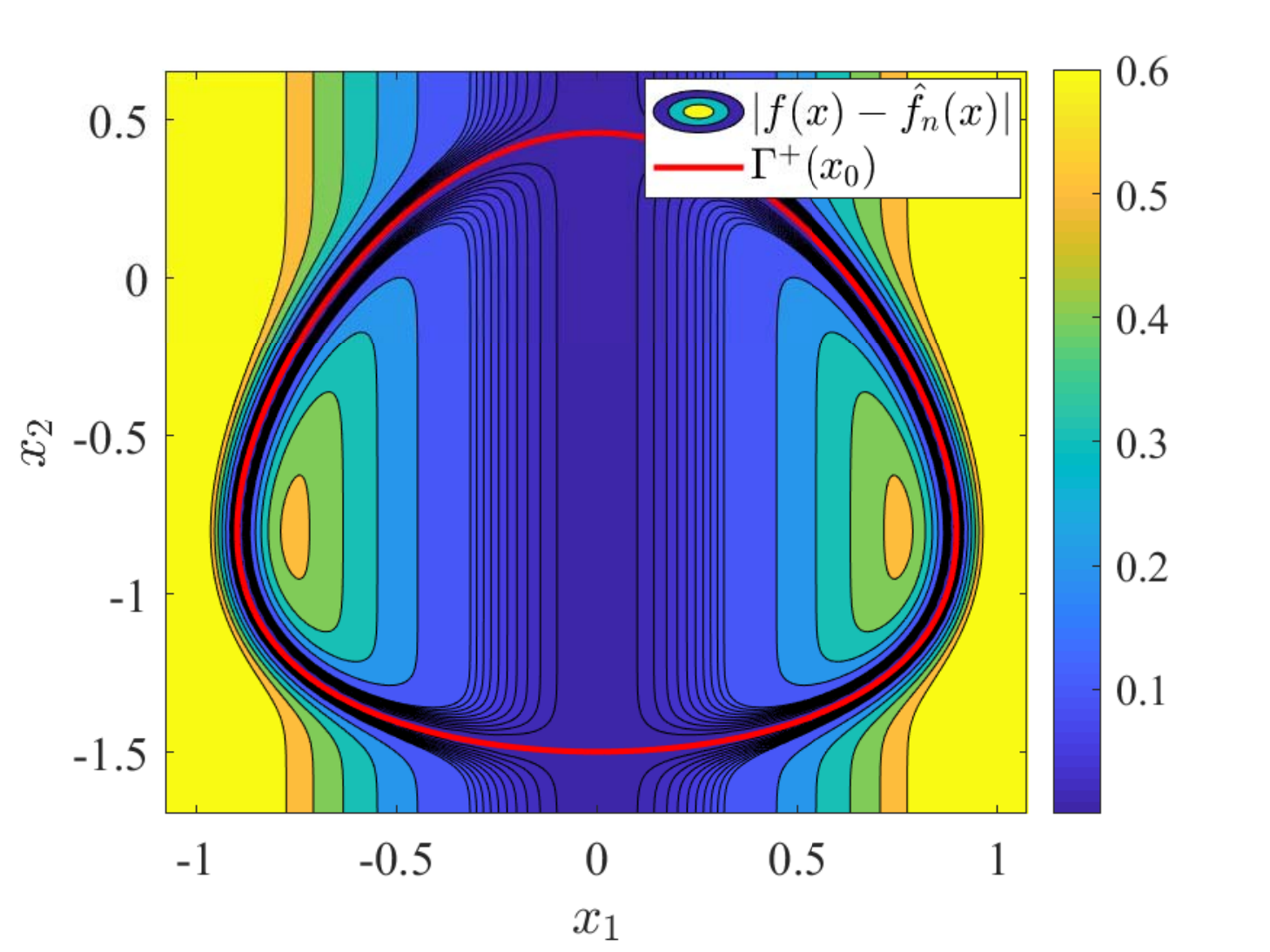}
    \caption{Contour of the estimate error over $X:=\mathbb{R}^2$}
    \label{fig:err_cont}
\end{figure}
The phase portraits of the original system in Eq. \ref{eq:sim_dyn} are shown in Fig. \ref{fig:phase_plot}. A first integral of the unknown system is
\begin{equation}
\label{eq:1st_int}
    \Phi(x):= (x_2 + x_1^2 - 0.5)e^{2x_2} = c.
\end{equation}
The stability of the system depends on the initial condition $(x_1(0),x_2(0))$. When the initial condition $x(0)$ is such that the constant $c<0$, the system is stable and the positive orbit $\Gamma^+(x(0))$ itself is the invariant manifold $\Omega:=\{x\in\mathbb{R}^2: \Phi(x) = c\}$. The manifold $\Omega$ is a smooth, one dimensional, regularly embedded submanifold in the phase space $\mathbb{R}^2$. In this example, we choose the trajectory for of which the constant $c=-0.1$ in Eq. \ref{eq:1st_int}. The samples $\Omega_n=\{\xi_i\}_{i=1}^N$ are taken uniformly along the manifold with respect to the intrinsic metric of the manifold $\Omega$. Although in practice, such sampling procedure cannot be accomplished without knowing the manifold \textit{a priori}, it is not difficult to picture that as $t\rightarrow\infty$, the set $\{x(t_i)\}_{i=1}^N$ gradually fills the manifold $\Omega$. The samples are used to construct the approximant RKHS $\Hn:=\spand{\Omega_n}$. The Sobolev-Mate\'rn kernel $\knl_\nu$ is used to induce the RKHS. The subscript $\nu$ denotes the order of the kernel. If $\nu>d/2$, then all the functions in the RKHS $H_{X}$ induced by $\knl_\nu$ over $X=\Rd$ also belong to every Sobolev space $W^{\tau,2}(\Rd)$ where $\tau>2\nu-d/2$ \cite{Fu}. The general expression of $\knl_\nu$ is defined using a Bessel function, but when $\nu=p+1/2,\,p\in\mathbb{N}$ the kernel has the following closed-form expressions 
\begin{align*}
    \knl_{3/2}(x,y) &= \left(1 + \frac{\sqrt{3}r}{l}\right)\exp{\left(-\frac{\sqrt{3}r}{l}\right)}, \\
    \knl_{5/2}(x,y) &= \left(1 + \frac{\sqrt{5}r}{l} + \frac{5r^2}{3l^2}\right)\exp{\left(-\frac{\sqrt{5}r}{l}\right)},
\end{align*}
where $r=\|x-y\|$, and $l$ is the scaling factor of length \cite{rasm}.

Fig. $\ref{fig:err_cont}$ shows the contour of the estimation error $|f(x)-\hat{f}_n(x)|$ in $\mathbb{R}^d$ using when $N=100$ and $\nu=5/2$. The result is as expected. The estimate of error in the unknown function is close to zero along the manifold $\Omega = \{x\in\mathbb{R}^2:\Phi(x) = -0.1\}$. The rate of convergence with respect to the number of samples $N$ is shown in Fig. \ref{fig:rate_32}-\ref{fig:rate_52}. Note that the manifold $\Omega$ is a closed curve, and the samples are taken uniformly in metric. As a result, the fill distance $h_{\Omega_n,\Omega}\sim N^{-1}$. With this in mind, by Corollary \ref{cor:sobolevrate} we have the following relationship
\begin{equation*}
    \|\Restr{\Omega}(\hat{f}(t) - \hat{f}_n(t))\|\wsub{\mu}\sim N^{-(s-\mu)}.
\end{equation*}
In this example, the set $\Omega$ is PE, so $\hat{f}(t)\rightarrow f$ over $\Omega$. The RKHS $H_X\hookrightarrow W^{\tau,2}(\mathbb{R}^2)$ where $\tau<2\nu-1$. Thus the space of restrictions $\Restr{\Omega}H_X\hookrightarrow W^{\tau-0.5,2}(\Omega)$, and $s\leq\tau-0.5<2\nu-1.5$. On the other hand, we must have $\mu\in[k/2,\lceil s\rceil-1]$ so that $W^{s,2}(\Omega)\hookrightarrow W^{\mu,2}(\Omega)\hookrightarrow C(\Omega)$. In this way, we have
\begin{equation*}
    \|f - \hat{f}_n(t)\|_{C(\Omega)} = \sup_{x\in\Omega}|f(x)-(\hat{f}_n(t))(x)|\sim N^{-(s-\mu)}.
\end{equation*}
From the analysis above, we obtain the rates of convergence for the $C(\Omega)$-norm. When $\nu=3/2$, the order $s-\mu\geq 1$. When $\nu=5/2$, the order $s-\mu\geq 2$. Taking the logarithms for both sides of the equation above, the values calculated above are the worst case of slope bounds in Fig. \ref{fig:rate_32}-\ref{fig:rate_52}. In both figures, the actual error curves are below the slope bounds, which validates the conclusions in Corollary \ref{cor:sobolevrate}. One assumption for the Theorem \ref{th:many0} to hold is that $h_{\Omega_n,\Omega}$ must be smaller than a threshold. This assumption may explain the flat error curve when $N\leq 30$.
\begin{figure}
    \centering
    \includegraphics[width=\linewidth]{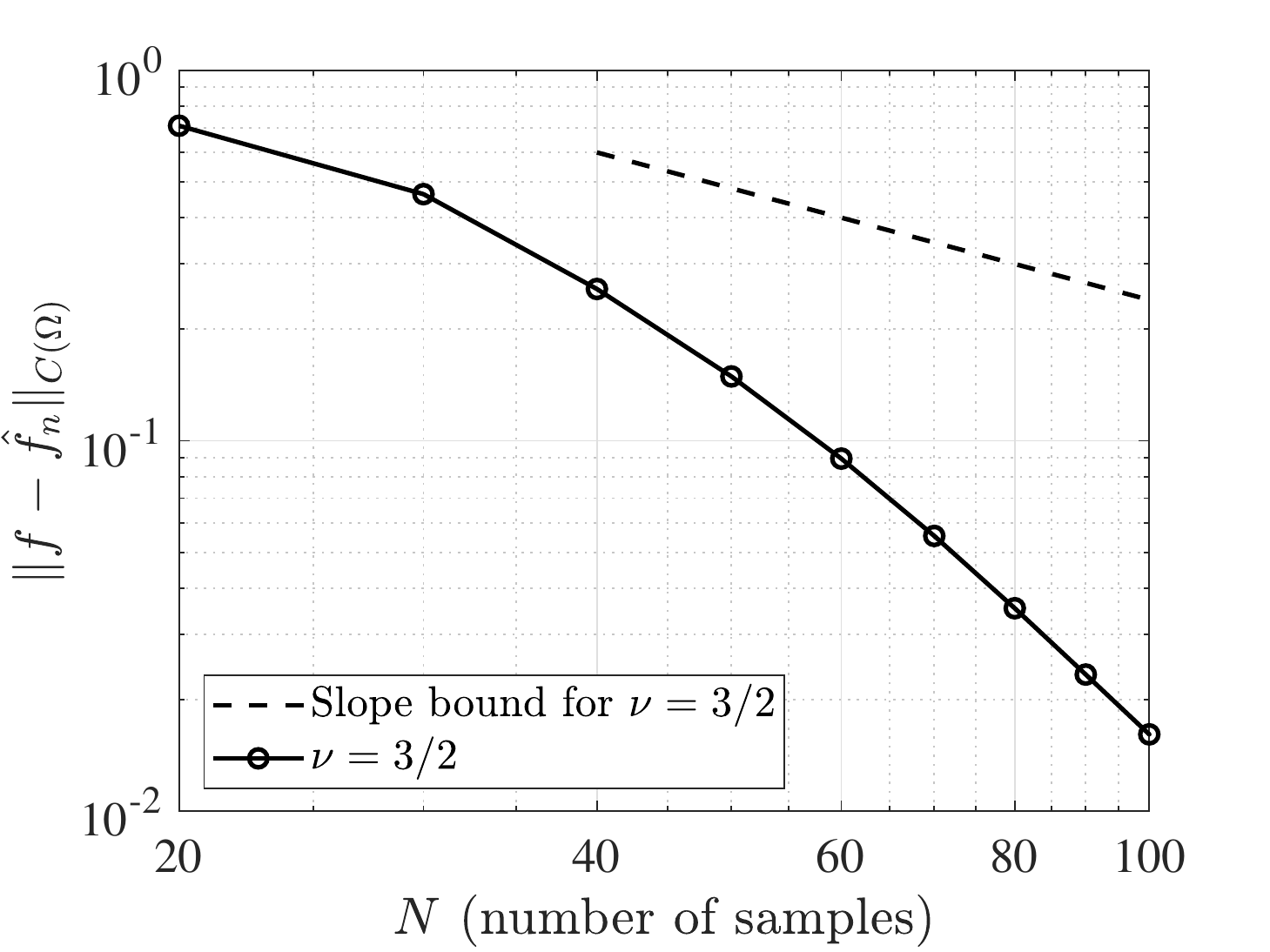}
    \caption{Rate of convergence for $\nu=3/2$.}
    \label{fig:rate_32}
\end{figure}
\begin{figure}
    \centering
    \includegraphics[width=\linewidth]{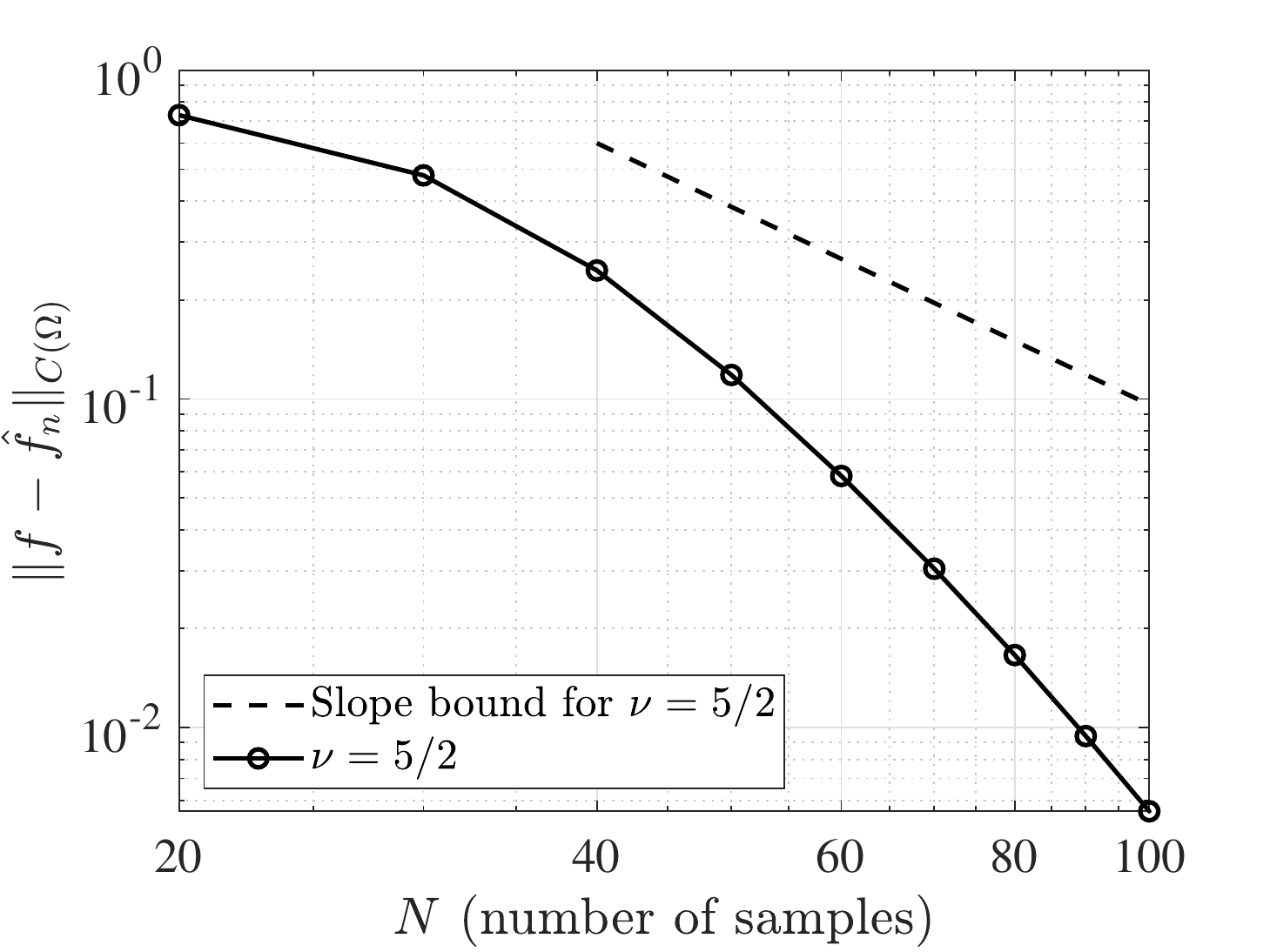}
    \caption{Rate of convergence for $\nu=5/2$.}
    \label{fig:rate_52}
\end{figure}

%
%
\section{Conclusions}
The RKHS embedding method constructs estimates of the unknown or uncertain functions that appear in types of ODEs in an infinite dimensional RKHS. This paper considers the practical problem of formulating finite dimensional approximations for this technique. The convergence of approximations is proven, and the rates of convergence are derived. By selecting the reproducing kernel that has algebraic decaying Fourier transform, the induced RKHS is embedded in or equivalent to a standard Sobolev space. The error equation of approximation is recast in the Sobolev space, and bounds on the error of interpolation in Sobolev spaces are applied to analyse the error of approximation. When the trajectory of the unknown system concentrates in a compact, regularly embedded submanifold of the state space, the rate of convergence for finite dimensional approximation is derived in terms of the fill distance of the samples. It is shown that as the samples becomes increasingly dense in the submanifold, the approximation error decays accordingly.

%
%
\bibliographystyle{IEEEtran}
\bibliography{rkhs_approx}

\end{document}